\documentclass[12pt, a4paper]{article}

\usepackage{authblk}
\author[1,2]{Masahiro Ikeda\thanks{ikeda@ist.osaka-u.ac.jp / masahiro.ikeda@a.riken.jp}}
\affil[1]{
Graduate School of Information Science and Technology,
The University of Osaka, 1-5,
Yamadaoka,
Suita-shi,
Osaka 565-0871,
Japan
}
\affil[2]{
Center for Advanced Intelligence Project RIKEN, 1-4-1, Nihonbashi, Chuo-ku, Tokyo 103-0027, Japan
}
\author[3]{Gustavo de Paula Ramos\thanks{Corresponding Author: gpramos@icmc.usp.br}}
\affil[3]{
Instituto de Ciências Matemáticas e de Computação,
Universidade de São Paulo,
Avenida Trabalhador São-Carlense 400,
13566-590,
São Carlos - SP,
Brazil
}

\usepackage[inline]{enumitem}
\setlist[enumerate,1]{label = (\roman*)}
\setlist[enumerate,2]{label = (\alph*)}

\usepackage[hidelinks]{hyperref}

\usepackage{amsmath}
\numberwithin{equation}{section}

\usepackage{amsthm}
\newtheorem{thm}{Theorem}[section]
 
\newtheorem{lem}[thm]{Lemma}
\newtheorem{cor}[thm]{Corollary}

\theoremstyle{definition}

\theoremstyle{remark}
\newtheorem{rmk}[thm]{Remark}

\usepackage{amsfonts, amssymb, mathrsfs, xcolor, mathtools}

\usepackage[backend=biber, sorting=none]{biblatex}
\AtEveryBibitem{\clearfield{note}}
\newcommand{\nat}{\mathbb{N}}
\newcommand{\real}{\mathbb{R}}
\newcommand{\complex}{\mathbb{C}}
\newcommand{\op}{- \Delta_\alpha}

\newcommand{\Sobolev}{H^1_\alpha}
\newcommand{\Sobolevstar}{H^{- 1}_\alpha}

\newcommand{\Ne}{\mathcal{N}_\omega}

\newcommand{\ev}{\ell_\alpha}
\newcommand{\aev}{\abs{\ev}}
\newcommand{\ef}{G_{\aev}}

\newcommand{\Sp}{\mathscr{S}}

\newcommand{\ACP}{\mathcal{A}_\omega}

\newcommand{\E}{\mathcal{E}}
\newcommand{\Action}{\mathcal{S}_\omega}

\newcommand{\eps}{\varepsilon}

\newcommand{\dif}{\ \mathrm{d}}

\newcommand{\iu}{\mathrm{i}}
\newcommand\odd{\mathrm{o}}
\newcommand{\even}{\mathrm{e}}

\DeclareMathOperator{\VS}{VS}
\DeclareMathOperator{\Dom}{D}

\DeclarePairedDelimiter{\abs}{\lvert}{\rvert}
\DeclarePairedDelimiter{\norm}{\lVert}{\rVert}
\DeclarePairedDelimiter{\parens}{(}{)}
\DeclarePairedDelimiter{\set}{\{}{\}}

\DeclarePairedDelimiter{\angles}{\langle}{\rangle}

\DeclarePairedDelimiter{\coi}{\lbrack}{\lbrack}

\DeclarePairedDelimiter{\ooi}{\rbrack}{\lbrack}

\title{Ground states of the defocusing nonlinear Schrödinger equation with a point interaction in dimensions 2 and 3}

\begin{document}

\maketitle

\begin{abstract}
This paper is concerned with ground states of the defocusing nonlinear Schrödinger equation with a point interaction,
\[
\iu \partial_t \psi
=
-\Delta_\alpha \psi + \psi |\psi|^{p - 2}
\quad \text{in} \quad
\mathbb{R} \times \mathbb{R}^N,
\]
where $- \Delta_\alpha$ denotes the Laplacian of point interaction centered at the origin with inverse s-wave scattering length $- 2 (N - 1) \pi \alpha$ and we suppose that either (i)
$N = 2$, $\alpha \in \real$ and $p > 2$ or (ii)
$N = 3$, $\alpha < 0$ and $2 < p < 3$. At sufficiently small masses, (i) we prove that this equation admits ground states, (ii) we obtain some qualitative properties of ground states and (iii) we obtain some results relating ground states with critical points of the associated action functional.
\end{abstract}

\tableofcontents

\section{Introduction}
This paper is concerned with ground states of the following \emph{defocusing nonlinear Schrödinger equation (NLSE) with a point interaction}:
\begin{equation}
\label{eqn:defocusing_NLSE}
\iu \partial_t \psi
=
\op \psi + \psi |\psi|^{p - 2}
\quad \text{in} \quad
\real \times \real^N,
\end{equation}
where $\op$ denotes the Laplacian of point interaction centered at the origin with inverse s-wave scattering length
$- 2 (N - 1) \pi \alpha$ and we suppose that either
\begin{equation}
\label{eqn:N=2}
N = 2, \quad \alpha \in \real \quad \text{and} \quad p > 2
\end{equation}
or
\begin{equation}
\label{eqn:N=3}
N = 3, \quad \alpha < 0 \quad \text{and} \quad 2 < p < 3.
\end{equation}

Before recalling the precise definition of $\op$, we anticipate that $\op$ is an $L^2$-self-adjoint operator which models the action of $- \Delta$ together with a potential concentrated at the origin in spatial dimensions $N \in \{1, 2, 3\}$. Even though the definition of $\op$ was proposed in the 1960s (see \cite{berezinRemarkSchrodingerEquation1961}), the study of nonlinear problems with a point interaction is much more recent.

To contextualize, let us recall a few papers concerned with the focusing NLSE with a point interaction,
\begin{equation}
\label{eqn:focusing_NLSE}
\iu \partial_t \psi = \op \psi - \psi \abs{\psi}^{p - 2}
\quad \text{in} \quad
\real \times \real^N.
\end{equation}
On one hand, this equation is relatively well understood when
$N = 1$, in which case the energy space associated with $\op$ is the usual Sobolev space $H^1 (\real)$. To cite a few papers, \cite{fukuizumiNonlinearSchrodingerEquation2008, fukuizumiStabilityStandingWaves2008, lecozInstabilityBoundStates2008} contain results about the orbital stability/instability of the unique positive solution; \cite{masakiAsymptoticStabilitySolitary2023} investigated the asymptotic stability; \cite{ohtaStrongInstabilityStanding2016} studied the strong instability and \cite{ikedaGlobalDynamicsStanding2017} contains results about the global dynamics below ground states (see also \cite{pavaInstabilityCnoidalpeakNLSd2012, pavaNonlinearSchrodingerEquation2013, pavaStabilityStandingWaves2017} for related results). On the other hand, the study of \eqref{eqn:focusing_NLSE} in higher dimensions $N \in \{2, 3\}$ started receiving increasing attention only in the last four years. More precisely, \cite{fukayaStabilityInstabilityStanding2022} proved the existence of action ground states and studied conditions for their orbital stability/instability in the case $N = 2$; \cite{adamiExistenceStructureRobustness2022, adamiGroundStatesPlanar2022} proved the existence and obtained qualitative properties of ground states and clarified the relation between the notions of ground states and action mininizers and \cite{fukayaUniquenessNondegeneracyGround2025} proved the uniqueness and nondegeneracy of the positive action ground state. Furthermore, more general semilinear problems with a point interaction in higher dimension $N \in \{2, 3\}$ were already considered in \cite{depaularamosGroundStatesPlanar2026, depaularamosMinimizersMassconstrainedFunctionals2026, georgievStandingWavesGlobal2024, michelangeliSingularHartreeEquation2021, pomponioNonlinearScalarField2025}.

Shifting the attention to the defocusing NLSE, let us highlight the difference introduced by the presence of a point interaction. On one hand, it is well known that the defocusing NLSE
\[
\iu \partial_t \psi
=
- \Delta \psi + \psi |\psi|^{p - 2}
\quad \text{in} \quad
\real \times \real^N
\]
does not admit nontrivial standing waves with finite energy (see \cite[Exercise 6.4]{fibichNonlinearSchrodingerEquation2015}). In fact, this result is a direct consequence of the fact that
$- \Delta$ is a positive-definite operator. On the other hand, it is no longer clear whether a similar result holds for \eqref{eqn:defocusing_NLSE} when either \eqref{eqn:N=2} or \eqref{eqn:N=3} is satisfied, in which case the Laplacian of point interaction $\op$ has a negative eigenvalue.

Problem \eqref{eqn:defocusing_NLSE} has received comparatively less attention than \eqref{eqn:focusing_NLSE}. A few results were already obtained for this equation in the case $N = 1$: \cite{kaminagaStabilityStandingWaves2009} proved the existence and orbital stability of action ground states and \cite{cuccagnaStabilitySmallSolitons2019} studied the asymptotic behavior of small solutions. In higher dimension, \cite{caspersDifferentApproachSingular1994} proved the existence of standing waves in the case $N = 3$ and, very recently, \cite{fukayaStandingWavesDefocusing2026} showed that there exists a unique positive standing wave which is radial and radially decreasing; furthermore, this preprint obtained the explicit form of the set of nontrivial standing waves for a fixed frequency
$\omega \in [0, \aev[$ in dimensions $N \in \set{2, 3}$.

To finish the brief literature review, we cite a few papers with results which are applicable to both the focusing and defocusing cases: \cite{masakiStabilitySmallSolitary2020} investigated the asymptotic stability in the case $N = 1$;
\cite{cacciapuotiWellPosednessNonlinear2021} studied the well-posedness of the associated Cauchy problem; \cite{cacciapuotiFailureScatteringNLSE2023} obtained conditions for failure of scattering and \cite{boniPointInteractionsSingular2026} explained the relation between standing waves of the NLSE in the punctured Euclidean space $\real^N \setminus \{0\}$ with solutions to the NLSE with a point interaction.

In this context, this paper is motivated by the absence of results about ground states of \eqref{eqn:defocusing_NLSE} in higher dimensions.

We proceed to a few preliminaries needed to recall the precise definition of $\op$. Consider a fixed $\lambda > 0$. We let
$G_\lambda \colon \real^N \setminus \{0\} \to ]0, \infty[$
be defined by
\[
G_\lambda \parens{x}
=
\begin{cases}
\displaystyle
\frac{1}{2 \pi}
K_0 \parens{\sqrt{\lambda} \abs{x}},
&\text{if} ~ N = 2;
\\
\\
\displaystyle
\frac{e^{- \sqrt{\lambda} \abs{x}}}{4 \pi \abs{x}},
&\text{if} ~ N = 3,
\end{cases}
\]
so that
\begin{equation}
\label{eqn:Glambda}
- \Delta G_\lambda + \lambda G_\lambda = \delta_0
\quad \text{in} \quad
\real^N
\end{equation}
in the sense of distributions, where $K_0$ denotes a modified Bessel function of the second kind (see \cite[Section 9.6]{abramowitzHandbookMathematicalFunctions1972}). A change of variable shows that
\begin{equation}
\label{eqn:integrability_of_Glambda}
G_\lambda \in L^r
\quad \text{for every} \quad
r \in
\begin{cases}
\displaystyle
\coi{2, \infty},
&\text{if} ~ N = 2;
\\
\coi{2, 3},
&\text{if} ~ N = 3.
\end{cases}
\end{equation}
We also let
\begin{equation}
\label{eqn:beta_alpha}
\beta_\alpha (\lambda)
=
\begin{cases}
\displaystyle
\alpha
+
\frac{\log \sqrt{\lambda} + \gamma - \log 2}{2 \pi},
&\text{if} ~ N = 2;
\\
\\
\displaystyle
\alpha
+
\frac{\sqrt{\lambda}}{4 \pi},
&\text{if} ~ N = 3,
\end{cases}
\end{equation}
where $\gamma$ denotes the Euler--Mascheroni constant.

Now, we recall the rigorous definition of $\op$, referring the reader to the classical treatise \cite{albeverioSolvableModelsQuantum1988} for more information (see also \cite{galloneSelfAdjointExtensionSchemes2023} for a more modern treatment). The family of operators
$\{\op\}_{\alpha \in ]- \infty, \infty]}$
is the family of self-adjoint extensions in $L^2$ of the closure of
$- \Delta|_{C_0^\infty (\real^N \setminus \{0\})}$.
When $\alpha = \infty$, we obtain the Friedrichs extension
$- \Delta_\infty = - \Delta|_{H^2}$. When
$\alpha \in \real$, $\op$ acts as
\[\op (\phi + q G_\lambda) = - \Delta \phi - q \lambda G_\lambda\]
for every $\phi + q G_\lambda \in \Dom (\op)$, where
\begin{equation}
\label{eqn:Dom}
\Dom (\op)
:=
\set*{
	\phi + q G_\lambda:
	\phi \in H^2, ~
	q \in \complex, ~
	\lambda > 0
	~ \text{and} ~
	\beta_\alpha (\lambda) q
	=
	\phi \parens{0}
}.
\end{equation}
An important aspect of $\Dom (\op)$ is that there exists a continuum of decompositions of its elements in the sense that given a fixed $u \in \Dom (\op)$, we can associate every $\lambda > 0$ with a decomposition of $u$ with the form of the elements in the RHS of \eqref{eqn:Dom}. This phenomenon also occurs in a slightly larger space, so we postpone this discussion to Remark \ref{rmk:decomposition}. We recall that the spectrum of $\op$ has the form
\[
\sigma \parens{\op}
=
\sigma_{\text{p}} \parens{\op}
\cup
\sigma_{\text{cont}} \parens{\op},
\]
where
\[\sigma_{\text{p}} \parens{\op} = \{\ev\},
\quad
\sigma_{\text{cont}} \parens{\op} = \coi{0, \infty}
\]
and
\begin{equation}
\label{eqn:ev}
\ev
:=
\begin{cases}
- 4 e^{- 4 \pi \alpha - 2 \gamma},
&\text{if} ~ N = 2 ~ \text{and} ~ \alpha \in \real;
\\
- (4 \pi \alpha)^2,
&\text{if} ~ N = 3 ~ \text{and} ~ \alpha < 0
\end{cases}
\end{equation}
is the unique eigenvalue of $\op$ (see \cite[Theorems 1.1.4 and 5.4]{albeverioSolvableModelsQuantum1988}). Furthermore, $\ev$ is a simple eigenvalue and $\ef$ is the eigenfunction associated with
$\ev$. The value of $\beta_\alpha (\lambda)$ varies as follows in function of $\ev$:
\begin{itemize}
\item
$\beta_\alpha (\lambda) > 0$ if, and only if,
$\lambda > \aev$;
\item
$\beta_\alpha (\lambda) = 0$ if, and only if,
$\lambda = \aev$;
\item
$\beta_\alpha (\lambda) < 0$ if, and only if,
$0 < \lambda < \aev$.
\end{itemize}
We also remark that $\op$ has no eigenvalue when $N = 3$ and
$\alpha \in [0, \infty[$, but this situation is outside the scope of this paper.

We are interested in \emph{standing waves} of \eqref{eqn:defocusing_NLSE}, that is, solutions of the form
\begin{equation}
\label{eqn:standing_wave}
\psi (t, x) = u (x) e^{\iu \omega t},
\end{equation}
where $u \in \Dom (\op)$ and $\omega \in \real$.
A function of the form \eqref{eqn:standing_wave} solves \eqref{eqn:defocusing_NLSE} precisely when
\begin{equation}
\label{eqn:stationary_defocusing_NLSE}
\op u + \omega u + u |u|^{p - 2} = 0
\quad \text{in} \quad
\real^N.
\end{equation}

Let us briefly introduce the adequate functional space to obtain a variational formulation of \eqref{eqn:stationary_defocusing_NLSE}, referring the reader to \cite[Section 1.15]{galloneSelfAdjointExtensionSchemes2023} for more details. Consider the vector space
\[
\VS (\Sobolev)
:=
\set*{
	\phi + q G_\lambda:
	\phi \in H^1, ~ q \in \complex ~ \text{and} ~ \lambda > 0
}
\]
and the following sesquilinear form defined on $\VS (\Sobolev)$:
\begin{align*}
h_\alpha (u_1, u_2)
:=
\int \nabla \phi_1 (x) \cdot \nabla \phi_2 (x) \dif x
+
\lambda \parens*{
	\angles{\phi_1, \phi_2}_{L^2} - \angles{u_1, u_2}_{L^2}
}
+
\beta_\alpha (\lambda) q_1 \overline{q_2}
\end{align*}
for every $u_j = \phi_j + q_j G_\lambda \in \Sobolev$
(for a proof that this quantity does not depend on $\lambda > 0$, see \cite[Lemma 2.4]{depaularamosMinimizersMassconstrainedFunctionals2026}). In this context, the aforementioned adequate Hilbert space is
\[
\Sobolev
:=
\parens*{\mathrm{VS} (\Sobolev), \angles{\cdot, \cdot}_{\Sobolev}}
\]
(also called the \emph{form domain} or \emph{energy space} associated with $\op$), where
\[
\angles{u_1, u_2}_{\Sobolev}
:=
h_\alpha (u_1, u_2) + \parens*{1 + \aev} \angles{u_1, u_2}_{L^2}
\]
for every $u_1, u_2 \in \Sobolev$. In view of \eqref{eqn:integrability_of_Glambda} and the Sobolev embeddings of $H^1$, it follows that $\Sobolev \hookrightarrow L^r$ for every
\[
r \in
\begin{cases}
[2, \infty[,
&\text{if} ~ N = 2;
\\
[2, 3[,
&\text{if} ~ N = 3.
\end{cases}
\]

\begin{rmk}[A continuum of decompositions of functions in $\Sobolev$]
\label{rmk:decomposition}
Suppose that $u \in H^1_\alpha$.
\begin{enumerate}
\item
There exists a unique $Q_\alpha (u) \in \complex$ (called the \emph{charge} of $u$) such that the following implication holds:
if $\phi_\lambda \in H^1$, $q \in \complex$ and $\lambda > 0$ are such that $u = \phi_\lambda + q G_\lambda$, then
$q = Q_\alpha (u)$.
\item
It follows we can associate every $\lambda > 0$ with a unique
$\phi_\lambda \in H^1$ such that
$u = \phi_\lambda + Q_\alpha (u) G_\lambda$.
\end{enumerate}
For more details, see \cite[Section 1.3]{depaularamosMinimizersMassconstrainedFunctionals2026} or \cite[Section 1.3.2]{depaularamosGroundStatesPlanar2026}.
\end{rmk}

At this point, we can finally define the considered notion of ground state. Consider the \emph{energy functional}
$\E \colon \Sobolev \to \real$
defined by
\begin{equation}
\label{eqn:energy_functional}
\E (u)
=
\frac{1}{2} h_\alpha (u, u)
+
\frac{1}{p} \|u\|_{L^p}^p.
\end{equation}
In this context, an \emph{ground state} of \eqref{eqn:defocusing_NLSE} with mass $\mu$ is a solution to the following minimization problem:
\begin{equation}
\label{eqn:minimization_problem}
\begin{cases}
\E (u) = E (\mu) := \inf \{\E (v): v \in \Sp (\mu)\};
\\
u \in \Sp (\mu),
\end{cases}
\end{equation}
where
\[
\Sp (\mu) := \set*{v \in \Sobolev: \|v\|_{L^2}^2 = \mu}.
\]
It is obvious that $E (\mu) > - \infty$. Indeed, the equality
$\inf \sigma \parens{- \Delta_\alpha} = \ev$
implies $\E (u) \geq \frac{\ev \mu}{2}$ for every
$u \in \Sp (\mu)$.

The following result is an important stepping stone to obtain our main theorem. More precisely, this lemma is proved by means of a careful study of the properties of the \emph{ground state energy function}
$E \colon [0, \infty[ \to ]- \infty, \infty]$
inspired by the arguments in \cite{bellazziniScalingPropertiesFunctionals2011, bellazziniStableStandingWaves2011}.

\begin{lem}
\label{lem:E_is_negative_and_SD}
Suppose that either \eqref{eqn:N=2} or \eqref{eqn:N=3} is satisfied. Then there exists
$\mu_0 = \mu_0 (N, \alpha, p) > 0$ such that $E|_{]0, \mu_0[}$ takes negative values and is strictly decreasing.
\end{lem}

Finally, our main result establishes the existence and some qualitative properties of ground states.

\begin{thm}
\label{thm:main}
Suppose that either \eqref{eqn:N=2} or \eqref{eqn:N=3} is satisfied. Suppose further that
$\mu_0 = \mu_0 (N, \alpha, p) > 0$
is furnished by Lemma \ref{lem:E_is_negative_and_SD} and
$0 < \mu < \mu_0$. The following conclusions are satisfied.
\begin{enumerate}
\item
\label{thm:main:i}
The NLSE \eqref{eqn:defocusing_NLSE} admits a ground state with mass $\mu$.
\item
\label{thm:main:ii}
If $u$ is a ground state of \eqref{eqn:defocusing_NLSE} with mass $\mu$, then it has the following properties.
\begin{enumerate}
\item
\label{thm:main:ii:a}
$u \in \Dom (\op)$ and \eqref{eqn:stationary_defocusing_NLSE} is satisfied with
$\omega := - (h_\alpha (u, u) + \|u\|_{L^p}^p)$.
\item
\label{thm:main:ii:b}
$Q_\alpha (u) \neq 0$ and
$u - Q_\alpha (u) G_\lambda \in H^2 \setminus \set{0}$
for every $\lambda > 0$.
\item
\label{thm:main:ii:c}
$u$ is radial.
\end{enumerate}
\end{enumerate}
\end{thm}

Let us comment on the proof of Theorem \ref{thm:main} \ref{thm:main:i}. Suppose that $(u_k)_{k \in \nat}$ is a minimizing sequence of the constrained energy functional $\E|_{\Sp (\mu)}$ and
$u_k \rightharpoonup u_\infty$ in $\Sobolev$ as $k \to \infty$. It is easy to prove that the energy functional $\E$ is weakly lower semicontinuous (see Lemma \ref{lem:weak_lower_semicontinuity}), so
\begin{equation}
\label{eqn:lower_energy}
\E (u_\infty) \leq \lim_{k \to \infty} \E (u_k) = E (\mu).
\end{equation}
The compactness of
$\set{u_k}_{k \in \nat \cup \{\infty\}}$
only fails when
$0 \leq \|u_\infty\|_{L^2}^2 < \mu$,
but these inequalities are not satisfied due to Lemma \ref{lem:E_is_negative_and_SD} and \eqref{eqn:lower_energy}, hence the existence of ground states.

Now, we comment on the proof of Theorem \ref{thm:main} \ref{thm:main:ii}. While it suffices to use Schwarz symmetrization to prove that ground states of the focusing equation \eqref{eqn:focusing_NLSE} are radial (see \cite{fukayaStabilityInstabilityStanding2022,adamiExistenceStructureRobustness2022,adamiGroundStatesPlanar2022}), this method does not seem to be applicable for the defocusing equation \eqref{eqn:defocusing_NLSE} due to the following inequality:
\[
\norm{\phi_\lambda^* + q G_\lambda}_{L^p}
\geq
\norm{\phi_\lambda + q G_\lambda}_{L^p}
\]
for every $\phi_\lambda + q G_\lambda \in \Sobolev$ (see \cite[Lemma 4.2]{fukayaStabilityInstabilityStanding2022}). As such, we use the fact that
$E|_{]0, \mu[}$ is strictly decreasing to prove \ref{thm:main:ii:c} by a comparison of levels of energy. On the other hand, it suffices to proceed as in \cite[Appendix A]{adamiGroundStatesPlanar2022} to prove \ref{thm:main:ii:a} and we only have to argue as in the proof of \cite[Proposition 5.2]{adamiGroundStatesPlanar2022} to prove \ref{thm:main:ii:b}.

An alternative way to obtain standing waves of \eqref{eqn:defocusing_NLSE} consists in looking for solutions to \eqref{eqn:stationary_defocusing_NLSE} with an \textit{a priori} known $\omega \in \real$, in which case it suffices to look for critical points of the \emph{action functional}
$\Action \colon \Sobolev \to \real$
defined by
\begin{equation}
\label{eqn:action_functional}
\begin{aligned}
\Action (u)
&=
\E (u) + \frac{\omega}{2} \norm{u}_{L^2}^2
\\
&=
\frac{1}{2} h_\alpha (u, u)
+
\frac{\omega}{2} \norm{u}_{L^2}^2
+
\frac{1}{p} \|u\|_{L^p}^p.
\end{aligned}
\end{equation}
In this situation, we denote the set of nontrivial critical points of $\Action$ by
\[
\ACP
:=
\set*{
	u \in \Sobolev \setminus \set{0}:
	\Action'(u) = 0
}.
\]
The following very recent result shows that, up to phase change, the action functional admits a unique nontrivial critical point when $0 < \omega < \aev$.

\begin{thm}[{\cite[Theorem 1.4]{fukayaStandingWavesDefocusing2026}}]
\label{thm:FOR}
Suppose that either \eqref{eqn:N=2} or \eqref{eqn:N=3} is satisfied. Suppose further that $0 < \omega < \aev$. Then there exists a unique positive, radial and radially decreasing function $u \in \Sobolev \setminus \set{0}$ such that
\[
\ACP
=
\set{e^{\iu \theta} u : \theta \in \real}.
\]
\end{thm}

Based on Theorem \ref{thm:FOR}, we obtain the following result that explores the relation between ground states of \eqref{eqn:defocusing_NLSE} and critical points of the associated action functional.

\begin{thm}
\label{thm:notions}
Suppose that either \eqref{eqn:N=2} or \eqref{eqn:N=3} is satisfied.
\begin{enumerate}
\item
\label{thm:notions:i}
There exists
$\mu_1 = \mu_1 (N, \alpha, p) > 0$
such that if $0 < \mu < \mu_1$ and $u$ is a ground state of \eqref{eqn:defocusing_NLSE} with mass $\mu$,
then \eqref{eqn:defocusing_NLSE} has a unique ground state up to phase change.
\item
\label{thm:notions:ii}
If $0 < \omega < \aev$ and $u \in \ACP$, then $u$ is a ground state of \eqref{eqn:defocusing_NLSE} with mass
$\mu := \norm{u}_{L^2}^2$.
\end{enumerate}
\end{thm}

Notice that Theorem \ref{thm:notions} does not furnish a complete answer to the question if ground states of \eqref{eqn:defocusing_NLSE} are precisely the elements of
$\ACP$ with $0 < \omega < \aev$ because it could happen that ground states with a sufficiently large mass $\mu$ are associated with a nonpositive Lagrange multiplier $\omega$. In particular, the usual approach of using a Poho\v{z}aev identity as in \cite{pomponioNonlinearScalarField2025} to determine the sign of $\omega$ does not seem to be applicable in the presence of a point interaction.

We remark that the conclusion of Theorem \ref{thm:FOR} can be extended to include the case $\omega = 0$ by considering a novel functional space instead of $\Sobolev$ (see \cite[Theorem 1.4]{fukayaStandingWavesDefocusing2026}). As such, it suffices to consider straightforward adaptations involving \cite[Theorem 1.5]{fukayaStandingWavesDefocusing2026} to extend Theorem \ref{thm:notions} \ref{thm:notions:ii} to the case
$\omega = 0$.

\subsection*{Notation and conventions}

We consider complex functional spaces whose elements are functions defined a.e. in $\real^N$ and we integrate over the entire space
$\real^N$ when the domain of integration is omitted.

\subsection*{Acknowledgement}
This study was financed, in part, by the São Paulo Research Foundation (FAPESP), Brazil, Process Number \#2024/20593-0. The first-named author (M.I.) is supported by JSPS, the Grant-in-Aid for Scientific Research (C) (No. 23K03174) and the Grant-in-Aid for Transformative Research Areas (B) (No. 25K24910).

\section{Preliminaries}

It follows from the Plancherel theorem that
\begin{equation}
\label{eqn:norm_of_Glambda}
\|G_\lambda\|_{L^2}^2
=
\begin{cases}
\frac{1}{4 \pi \lambda},
&\text{if} ~ N = 2;
\\
\frac{1}{8 \pi \sqrt{\lambda}},
&\text{if} ~ N = 3.
\end{cases}
\end{equation}
A routine computation confirms the following characterization of weak convergence in $\Sobolev$.

\begin{lem}
\label{lem:weak_convergence}
Suppose that
$\set{u_k}_{k \in \nat \cup \set{\infty}} \subset \Sobolev$.
Consider a fixed $\lambda > 0$ and given
$k \in \nat \cup \set{\infty}$, let
$\phi_k = u_k - Q (u_k) G_\lambda \in H^1$.
Then $u_k \rightharpoonup u_\infty$ in $\Sobolev$ as
$k \to \infty$ if, and only if,
$\phi_k \rightharpoonup \phi_\infty$ in $H^1$ and
$Q (u_k) \to Q (u_\infty)$ as $k \to \infty$.
\end{lem}

In the following lemmata, we obtain a few properties of the energy functional $\E$. Our first goal is to obtain a coercivity-type result.

\begin{lem}
\label{lem:coercivity}
Suppose that $(\mu_k)_{k \in \nat}$ is a bounded sequence of positive numbers and
$\{u_k\}_{k \in \nat} \subset \Sobolev$
is such that
\begin{enumerate}
\item
\label{lem:coercivity:i}
$\|u_k\|_{\Sobolev} \to \infty$ as $k \to \infty$;
\item
\label{lem:coercivity:ii}
given $k \in \nat$, it holds that $u_k \in \Sp \parens{\mu_k}$.
\end{enumerate}
Then $\E \parens{u_k} \to \infty$ as $k \to \infty$. In particular, the constrained energy functional $\E|_{\Sp (\mu)}$ is coercive for every $\mu > 0$.
\end{lem}
\begin{proof}
Fix $\lambda > \aev$, so that
$\beta_\alpha (\lambda) > 0$. To simplify the notation, let
\[\phi_k = u_k - Q_\alpha (u_k) G_\lambda \in H^1\]
for every $k \in \nat$. It follows from \ref{lem:coercivity:i} that either
\[
\|\nabla \phi_k\|_{L^2}^2 + \lambda \|\phi_k\|_{L^2}^2
\xrightarrow[k \to \infty]{}
\infty
\]
or $|Q_\alpha (u_k)| \to \infty$ as $k \to \infty$. The set
$\{\mu_k\}_{k \in \nat}$ is bounded, so
\begin{align*}
\E (u_k)
&=
\frac{1}{2}
\parens*{
	\|\nabla \phi_k\|_{L^2}^2
	+
	\lambda \|\phi_k\|_{L^2}^2
	+
	\beta_\alpha (\lambda) \abs*{Q_\alpha (u_k)}^2
}
-
\frac{\lambda \mu_k}{2}
+
\frac{1}{p}
\norm{u_k}_{L^p}^p
\\
&\xrightarrow[k \to \infty]{}
\infty.
\end{align*}
\end{proof}

To finish, we show that $\E$ is weakly lower semicontinuous.

\begin{lem}
\label{lem:weak_lower_semicontinuity}
The energy functional
$\E \colon \Sobolev \to \real$
is weakly lower semicontinuous.
\end{lem}
\begin{proof}
Suppose that
\[
\set*{u_k = \phi_k + q_k \ef}_
	{k \in \nat \cup \{\infty\}}
\subset
\Sobolev
\]
is such that $u_k \rightharpoonup u_\infty$ in $\Sobolev$ as
$k \to \infty$. By definition,
\begin{equation}
\label{lem:weak_lower_semicontinuity:aux:0}
\E (u_k)
=
\frac{1}{2} \parens*{
	\|\nabla \phi_k\|_{L^2}^2
	+
	2 \ev
	\angles*{\Re (\overline{q_k} \phi_k), \ef}_{L^2}
	+
	\ev |q_k|^2
}
+
\frac{1}{p} \|u_k\|_{L^p}^p.
\end{equation}
On one hand, Lemma \ref{lem:weak_convergence} shows that
$q_k \to q_\infty$ and $\phi_k \rightharpoonup \phi_\infty$ in $H^1$ as $k \to \infty$, so it follows that
\begin{equation}
\label{lem:weak_lower_semicontinuity:aux:1}
2 \angles*{\Re (\overline{q_k} \phi_k), \ef}_{L^2}
+
|q_k|^2
\xrightarrow[k \to \infty]{}
2
\angles*{
	\Re (\overline{q_\infty} \phi_\infty), \ef
}_{L^2}
+
|q_\infty|^2.
\end{equation}
On the other hand, the functionals
\[
H^1 \ni \psi \mapsto \|\nabla \psi\|_{L^2} \in [0, \infty[
\quad \text{and} \quad
L^p \ni \psi \mapsto \|\psi\|_{L^p} \in [0, \infty[
\]
are weakly lower semicontinuous, so
\begin{equation}
\label{lem:weak_lower_semicontinuity:aux:2}
\frac{1}{2}
\|\nabla \phi_\infty\|_{L^2}^2
+
\frac{1}{p}
\|u_\infty\|_{L^p}^p
\leq
\liminf_{k \to \infty} \parens*{
	\frac{1}{2}
	\|\nabla \phi_k\|_{L^2}^2
	+
	\frac{1}{p}
	\|u_k\|_{L^p}^p
}.
\end{equation}
In view of \eqref{lem:weak_lower_semicontinuity:aux:0}--\eqref{lem:weak_lower_semicontinuity:aux:2}, we deduce that
$
\E (u_\infty)
\leq
\liminf_{k \to \infty} \E (u_k)
$.
\end{proof}

\section{Properties of the ground state energy function $E$}
\label{sect:E_is_negative_and_SD}

The goal of this section is to prove Lemma \ref{lem:E_is_negative_and_SD}. We begin by developing a few preliminary results which will be used to prove the aforementioned lemma. The first one is the following important limit.

\begin{lem}
\label{lem:limit}
The following limit holds:
\begin{equation}
\label{lem:limit:i}
\frac{2 E (\mu)}{\aev \mu}
\xrightarrow[\mu \to 0]{}
- 1.
\end{equation}
\end{lem}
\begin{proof}
On one hand, we have
\[
\E (u)
=
\frac{1}{2} h_\alpha (u, u)
+
\frac{1}{p} \|u\|_{L^p}^p
\geq
\frac{1}{2} h_\alpha (u, u)
\geq
- \frac{\aev \mu}{2}
\]
for every $u \in \Sp (\mu)$ because
$\ev = \inf \sigma (\op)$. In particular, we deduce that
\begin{equation}
\label{lem:limit:i:1}
\liminf_{\mu \to 0} \frac{2 E (\mu)}{\aev \mu} \geq - 1.
\end{equation}
On the other hand, it is clear that
\[
\frac{1}{\mu}
\E \parens*{
	\frac{\sqrt{\mu}}{\|\ef\|_{L^2}}
	\ef
}
=
\frac{\ev}{2}
+
\frac{\mu^{\frac{p - 2}{2}}}{p}
\parens*{
	\frac{\|\ef\|_{L^p}}{\|\ef\|_{L^2}}
}^p
\xrightarrow[\mu \to 0]{}
\frac{\ev}{2},
\]
so
\begin{equation}
\label{lem:limit:i:2}
\limsup_{\mu \to 0} \frac{2 E (\mu)}{\aev \mu} \leq - 1.
\end{equation}
Finally, \eqref{lem:limit:i} follows directly from \eqref{lem:limit:i:1} and \eqref{lem:limit:i:2}.
\end{proof}

The previous limit implies the existence of
$\mu_0 > 0$ such that $E|_{]0, \mu_0[}$ takes negative values.

\begin{cor}
\label{cor:E<0}
There exists
$\mu_0 = \mu_0 (N, \alpha, p) > 0$
such that if
$0 < \mu < \mu_0$, then $E (\mu) < 0$.
\end{cor}

Our strategy is to use the following elementary result to prove the existence of $\mu_0 > 0$ as in Lemma \ref{lem:E_is_negative_and_SD}.

\begin{lem}
\label{lem:elementary}
Suppose that $k > 0$, $f \colon ]0, k[ \to \real$ is lower semicontinuous and we can associate every
$a \in ]0, k[$ with $t_a > 1$ and $\eps_a > 0$ such that
\[f (t a) \leq f (a) - \eps_a (t - 1)\]
for every $t \in [1, t_a]$. Then $f$ is strictly decreasing.
\end{lem}
\begin{proof}
~\paragraph{Step 1: beginning of the argument.}
Let us prove that if $0 < a < b < k$, then $f (a) > f (b)$. Suppose that $0 < a < k$ and let
\begin{multline*}
d
=
\sup \left\{
	c \in ]a, k[:
	f (a) > f (b) ~ \text{for every} ~ b \in ]a, c[
\right.
\\
\left.
	\text{and} ~
	f (c) \leq f (a) - \eps_a (t_a - 1)
\right\}.
\end{multline*}
We want to show that $d = k$. By contradiction, suppose that $d < k$. Due to the condition in the statement of the lemma, $d \geq t_a a > a$. In view of the definition of $d$ and the fact that $f$ is lower semicontinuous, we have
\begin{equation}
\label{lem:elementary:aux:3}
f (d) \leq f (a) - \eps_a (t_a - 1).
\end{equation}
To obtain a contradiction with the definition of $d$, we only have to prove that
\begin{equation}
\label{lem:elementary:aux:1}
f (a) > f (b) ~ \text{for every} ~ b \in ]a, t_d d[
\end{equation}
and
\begin{equation}
\label{lem:elementary:aux:2}
f (t_d d) < f (a) - \eps_a (t_a - 1).
\end{equation}

\paragraph{Step 2: proof of \eqref{lem:elementary:aux:1}.}
It follows from the definition of $d$ that
$f (a) > f (b)$ for every $b \in ]a, d[$. As such, we only have to prove that $f (a) > f (b)$ for every $b \in [d, t_d d[$. On one hand, it follows from the definition of $t_d$ that $f (d) > f (b)$ for every $b \in ]d, t_d d[$. On the other hand, it follows from \eqref{lem:elementary:aux:3} that $f (a) > f (d)$. We deduce that $f (a) > f (b)$ for every $b \in [d, t_d d[$.

\paragraph{Step 3: proof of \eqref{lem:elementary:aux:2}.}
The inequality
\[f (t_d d) \leq f (d) - \eps_d (t_d - 1) < f (d)\]
holds by hypothesis, so the result follows from \eqref{lem:elementary:aux:3}.
\end{proof}

The following lemmata verify that there exists $k > 0$ such that we can apply Lemma \ref{lem:elementary} in the case
$f := E|_{]0, k[}$. First, we address lower semicontinuity.

\begin{lem}
\label{lem:E_is_lower_semicontinuous}
If $\mu_0$ is furnished by Corollary \ref{cor:E<0}, then $E|_{]0, \mu_0[}$ is lower semicontinuous.
\end{lem}
\begin{proof}
Let us follow the arguments in \cite[Step 4 at Proof of Theorem 4.1]{bellazziniScalingPropertiesFunctionals2011}. Suppose that
$
\set{\mu_k}_{k \in \nat \cup \{\infty\}}
\subset
]0, \mu_0[
$
is such that
$\lim_{k \to \infty} \mu_k = \mu_\infty$. As $E|_{]0, \mu_0[}$ takes negative values, we can associate each
$k \in \nat$ with a
$u_k \in \Sp \parens{\mu_k}$ such that
\[\E \parens{u_k} < E (\mu_k) + \frac{1}{k} < \frac{1}{k}.\]
By definition,
\begin{align}
E (\mu_\infty)
&\leq
\E \parens*{\sqrt{\frac{\mu_\infty}{\mu_k}} u_k}
-
\E (u_k)
+
\E (u_k)
\nonumber
\\
&<
\frac{1}{2}
\parens*{\frac{\mu_\infty}{\mu_k} - 1}
h_\alpha (u_k, u_k)
+
\frac{1}{p}
\parens*{\parens*{\frac{\mu_\infty}{\mu_k}}^{\frac{p}{2}} - 1}
\norm{u_k}_{L^p}^p
+
E (\mu_k)
+
\frac{1}{k}.
\label{lem:E_is_lower_semicontinuous:aux:1}
\end{align}
As $\sup_{k \in \nat} E (\mu_k) \leq 0$, it follows from Lemma \ref{lem:coercivity} that $\{u_k\}_{k \in \nat}$ is bounded in $\Sobolev$. As such, we obtain
\begin{equation}
\label{lem:E_is_lower_semicontinuous:aux:2}
\frac{1}{2}
\parens*{\frac{\mu_\infty}{\mu_k} - 1}
h_{\alpha} (u_k, u_k)
+
\frac{1}{p}
\parens*{\parens*{\frac{\mu_\infty}{\mu_k}}^{\frac{p}{2}} - 1}
\norm{u_k}_{L^p}^p
\xrightarrow[k \to \infty]{}
0.
\end{equation}
In view of \eqref{lem:E_is_lower_semicontinuous:aux:1} and \eqref{lem:E_is_lower_semicontinuous:aux:2}, we deduce that
$E (\mu_\infty) \leq \liminf_{k \to \infty} E (\mu_k)$.
\end{proof}

Now, we want to prove that the remaining condition in Lemma \ref{lem:elementary} is also satisfied.

\begin{lem}
\label{lem:eps_mu}
There exists $\mu_0 = \mu_0 (N, \alpha, p) > 0$ such that we can associate each $\mu \in ]0, \mu_0[$ with
$t_\mu > 1$ and $\eps_\mu > 0$ such that
\[
E (t \mu) \leq E (\mu) - \eps_\mu (t - 1)
\]
for every $t \in [1, t_\mu]$.
\end{lem}
\begin{proof}
In view of Lemma \ref{lem:limit}, we can fix
$\mu_0 = \mu_0 (N, \alpha, p) > 0$ such that if
$0 < \mu < \mu_0$, then
\begin{equation}
\label{lem:eps_mu:2}
E \parens{\mu}
+
\frac{p - 2}{2}
\parens*{E \parens{\mu} + \frac{\aev \mu}{2}}
<
0.
\end{equation}
Suppose that $0 < \mu < \mu_0$ and let $(u_k)_{k \in \nat}$ denote a minimizing sequence of the constrained functional
$\E|_{\Sp (\mu)}$. To simplify the notation, let
\[
q_k = Q_\alpha (u_k) \in \complex
\quad \text{and} \quad
\phi_k = u_k - q_k \ef \in H^1
\]
for every $k \in \nat$. Given $t > 1$, we let
\[u_{k, t} \parens{x} = \sqrt{t} u_k \parens{x}\]
for a.e. $x \in \real^N$, so that
$u_{k, t} \in \Sp (t \mu)$. By definition,
\[
E (t \mu)
\leq
\E (u_{k, t})
=
\frac{t}{2} h_\alpha \parens{u_k, u_k}
+
\frac{t^{\frac{p}{2}}}{p}
\norm{u_k}_{L^p}^p.
\]
By summing and subtracting
$\frac{t}{p} \norm{u_k}_{L^p}^p$,
we obtain
\begin{equation}
\label{lem:eps_mu:3}
E (t \mu)
\leq
t \E \parens{u_k}
+
\frac{1}{p}
\parens{t^{\frac{p}{2}} - t}
\norm{u_k}_{L^p}^p.
\end{equation}
As $\ev = \inf \sigma \parens{\op}$, we deduce that
$h_\alpha \parens{u_k, u_k} + \aev \mu \geq 0$.
Therefore,
\begin{equation}
\label{lem:eps_mu:4}
\frac{1}{p} \norm{u_k}_{L^p}^p
\leq
\frac{1}{2} h_\alpha \parens{u_k, u_k}
+
\frac{\aev \mu}{2}
+
\frac{1}{p} \norm{u_k}_{L^p}^p
=
\E \parens{u_k} + \frac{\aev \mu}{2}.
\end{equation}
In view of \eqref{lem:eps_mu:3} and \eqref{lem:eps_mu:4}, we obtain
\[
E \parens{t \mu}
\leq
t \E \parens{u_k}
+
\parens{t^{\frac{p}{2}} - t}
\parens*{\E \parens{u_k} + \frac{\aev \mu}{2}}.
\]
A passage to the limit shows that
$E \parens{t \mu} \leq F_\mu \parens{t}$,
where the function
$F_\mu \colon [1, \infty[ \to \real$
is defined by
\[
F_\mu \parens{t}
=
t E \parens{\mu}
+
\parens{t^{\frac{p}{2}} - t}
\parens*{E \parens{\mu} + \frac{\aev \mu}{2}}.
\]
It is clear that $F_\mu \parens{1} = E \parens{\mu}$, so we only have to prove that $F_\mu ' (1) < 0$. Clearly,
\[
F_\mu' \parens{1}
=
E \parens{\mu}
+
\frac{p - 2}{2} 
\parens*{E \parens{\mu} + \frac{\aev \mu}{2}}.
\]
Finally, the inequality $F_\mu' (1) < 0$ follows from \eqref{lem:eps_mu:2}.
\end{proof}

Finally, we proceed to the proof of Lemma \ref{lem:E_is_negative_and_SD}.

\begin{proof}[Proof of Lemma \ref{lem:E_is_negative_and_SD}]
In view of Corollary \ref{cor:E<0}, Lemmata \ref{lem:E_is_lower_semicontinuous} and \ref{lem:eps_mu}, we can fix $\mu_0 = \mu_0 (p) > 0$ such that $E|_{]0, \mu_0[}$ takes negative values and the conditions in Lemma \ref{lem:elementary} are satisfied in the case $k := \mu_0$ and
$f := E|_{]0, \mu_0[}$. As such, $E|_{]0, \mu_0[}$ is strictly decreasing due to Lemma \ref{lem:elementary}.
\end{proof}

\section{Proof of Theorem \ref{thm:main}}
\subsection{Proof of \ref{thm:main:i}}

Let $(u_k)_{k \in \nat}$ denote a minimizing sequence of
$\E|_{\Sp \parens{\mu}}$.

\paragraph{Step 1: weak convergence up to subsequence.}
It follows from Lemma \ref{lem:coercivity} that
$(u_k)_{k \in \nat}$ is bounded in $\Sobolev$. In particular, there exists $u_\infty \in \Sobolev$ such that, up to subsequence,
$u_k \rightharpoonup u_\infty$ in $\Sobolev$ as $k \to \infty$. In view of Lemma \ref{lem:weak_lower_semicontinuity}, we deduce that
\begin{equation}
\label{thm:main:i:aux:1}
\E (u_\infty)
\leq
\liminf_{k \to \infty} \E (u_k)
=
\lim_{k \to \infty} \E (u_k)
=
E (\mu)
<
0.
\end{equation}

\paragraph{Step 2: $u_\infty \in \Sp (\mu)$.}
On one hand, the restricted function
$E|_{]0, \mu_0[}$ is strictly decreasing and $\E (0) = 0$, so \eqref{thm:main:i:aux:1} implies
\begin{equation}
\label{thm:main:i:aux:2}
\|u_\infty\|_{L^2}^2 \geq \mu.
\end{equation}
On the other hand, the weak convergence
$u_k \rightharpoonup u_\infty$ in $L^2$ as $k \to \infty$ implies
\begin{equation}
\label{thm:main:i:aux:3}
\|u_\infty\|_{L^2}^2 \leq \mu.
\end{equation}
In view of \eqref{thm:main:i:aux:2} and \eqref{thm:main:i:aux:3}, we deduce that $u \in \Sp (\mu)$.

\paragraph{Step 3: conclusion.}
As $\E (u_\infty) \leq E (\mu)$ and $u_\infty \in \Sp (\mu)$, we deduce that $\E (u_\infty) = E (\mu)$. It follows that $u_\infty$ is a ground state of \eqref{eqn:defocusing_NLSE} with mass $\mu$.
\qed

\subsection{Proof of \ref{thm:main:ii}}

\paragraph{Proof of \ref{thm:main:ii:a}.}
It suffices to argue as in \cite[Appendix A]{adamiGroundStatesPlanar2022} or \cite[Proof of Proposition 1.6]{depaularamosGroundStatesPlanar2026}.
\qed

\paragraph{Proof of \ref{thm:main:ii:b}.}
\subparagraph{Proof that $Q_\alpha (u) \neq 0$.}
By contradiction, suppose that $Q_\alpha (u) = 0$. On one hand,
\[
\E (u)
=
\frac{1}{2} \|\nabla u\|_{L^2}^2
+
\frac{1}{p} \|u\|_{L^p}^p
>
0
\]
because $u \in H^1 \setminus \{0\}$. On the other hand,
$\E (u) = E (\mu) < 0$. We just obtained a contradiction, so we deduce that $Q_\alpha (u) \neq 0$.
\qed

\subparagraph{Proof that $u - Q_\alpha (u) G_\lambda \in H^1 \setminus \{0\}$ for every $\lambda > 0$.}
Let us argue as in the proof of \cite[Proposition 5.2]{adamiGroundStatesPlanar2022}. By contradiction, suppose that $\lambda > 0$ is such that
\[\phi_\lambda := u - Q_\alpha (u) G_\lambda \equiv 0.\]
It follows from \ref{thm:main:ii:a} that $u \in \Dom (\op)$ and we just proved that $Q_\alpha (u) \neq 0$, so
$\beta_\alpha (\lambda) = 0$. That is,
$\lambda = \aev$. As \eqref{eqn:stationary_defocusing_NLSE} is satisfied, we deduce that
\[
\omega
+
\ev
+
\abs*{Q_\alpha (u)}^{p - 2} \ef (x)^{p - 2}
=
0
\]
for a.e. $x \in \real^2$, which is clearly not possible.
\qed

\paragraph{Proof of \ref{thm:main:ii:c}.}
\subparagraph{Step 1: strategy of the proof.}
It is elementary that linear reflections on $\real^N$ (i.e., reflections with respect to a linear subspace of $\real^N$ with dimension $N - 1$) generate the orthogonal group $\mathrm{O} (N)$. As such, it suffices to prove that if
$T \colon \real^N \to \real^N$ is a linear reflection, then
$u (x) = u (T x)$ for a.e. $x \in \real^N$.

\subparagraph{Step 2: fixing the terminology and notation.}
Fix a linear reflection $T \colon \real^N \to \real^N$. If $g$ is a function defined a.e. on $\real^N$, we respectively let
$g_\even, g_\odd$ denote its \emph{even part with respect to $T$} and its \emph{odd part w.r.t. $T$}, i.e.,
\[
g_\even (x)
:=
\frac{g (x) + g (T x)}{2}
\quad \text{and} \quad
g_\odd (x)
:=
\frac{g (x) - g (T x)}{2}
\]
for a.e. $x \in \real^N$. Notice that
\begin{enumerate}[label=(\alph*)]
\item
$g(x) = g_\even (x) + g_\odd (x)$;
\item
$g_\even (x) = g_\even (T x)$ ($g_\even$ is \emph{even w.r.t. $T$}) and
\item
$g_\odd (x) = - g_\odd (T x)$ ($g_\odd$ is \emph{odd w.r.t. $T$}).
\end{enumerate}
for a.e. $x \in \real^N$. To simplify the notation, let
\[
q = Q_\alpha (u)
\quad \text{and} \quad
\phi = u - q \ef \in H^1.
\]

\subparagraph{Step 3: $\|u_\even\|_{L^r}^r \leq \|u\|_{L^r}^r$ for $r \in \set{2, p}$.}
Let
\[A_T = \set{x \in \real^N: Tx = x}.\]
In fact, $A_T$ is a linear subspace of $\real^N$ with dimension
$N - 1$, so we can fix
$y_T \in A_T^\perp \setminus \set{0}$
and we can write $\real^N$ as the disjoint union
$\real^N = O_T^- \sqcup A_T \sqcup O_T^+$,
where
\[
O_T^+
:=
\set*{
	x \in \real^N: x \cdot y_T > 0
}
\quad \text{and} \quad
O_T^-
:=
\set*{
	x \in \real^N: x \cdot y_T < 0
}.
\]
The function
\[
\complex \ni z \mapsto |z|^r \in [0, \infty[
\]
is convex, so
\begin{equation}
\label{thm:main:ii:c:1}
\abs{u_\even}^r
\leq
\frac{
	\abs{u_\even + \phi_\odd}^r
	+
	\abs{u_\even - \phi_\odd}^r
}{2}
\quad \text{a.e. in} \quad
\real^N.
\end{equation}
Notice that $u_\even = \phi_\even + \ef$ and $u_\odd = \phi_\odd$ because the Green's function $\ef$ is even with respect to $T$. Therefore,
\begin{align*}
\norm{u}_{L^r}^r - \norm{u_\even}_{L^r}^r
&=
\int_{O_T^+}
	\abs*{u_\even (x) + \phi_\odd (x)}^r
	-
	\abs*{u_\even (x)}^r
\dif x
\\
&+
\int_{O_T^-}
	\abs*{u_\even (x) + \phi_\odd (x)}^r
	-
	\abs*{u_\even (x)}^r
\dif x.
\end{align*}
In view of the facts that $\phi_\odd$ is odd w.r.t. $T$, $u_\even$ is even w.r.t. $T$ and
$T (O_T^+) = O_T^-$, we obtain
\begin{align*}
\norm{u}_{L^r}^r - \norm{u_\even}_{L^r}^r
&=
\int_{O_T^+}
	\abs*{u_\even (x) + \phi_\odd (x)}´^r
	-
	\abs*{u_\even (x)}^r
\dif x
\\
&+
\int_{O_T^+}
	\abs*{u_\even (x) - \phi_\odd (x)}^r
	-
	\abs*{u_\even (x)}^r
\dif x
\end{align*}
As such, the inequality
$\|u\|_{L^r}^r \geq \|u_\even\|_{L^r}^r$
follows from \eqref{thm:main:ii:c:1}.

\subparagraph{Step 4: $\E (u_\even) \leq E (\mu)$.}
It is easy to verify that $\nabla \phi_\even$ is odd w.r.t. $T$ and $\nabla \phi_\odd$ is even w.r.t. $T$. In particular, we deduce that
\[
h_\alpha (u, u)
=
\|\nabla \phi_\even\|_{L^2}^2
+
\|\nabla \phi_\odd\|_{L^2}^2
-
2 \aev \angles*{\phi_\even, q \ef}_{L^2}
+
\ev |q|^2.
\]
and
\[
h_\alpha (u_\even, u_\even)
=
\|\nabla \phi_\even\|_{L^2}^2
-
2 \aev \angles*{\phi_\even, q \ef}_{L^2}
+
\ev |q|^2.
\]
Therefore, $h_\alpha (u_\even, u_\even) \leq h_\alpha (u, u)$. Due to this inequality and the inequality in Step 3, we deduce that
$\E (u_\even) \leq \E (u) = E (\mu)$.

\subparagraph{Step 5: conclusion.}
Let $\nu = \norm{u_\even}_{L^2}^2$. As
$\E (u_\even) \leq E (\mu) < 0$, we deduce that $\nu > 0$.
On one hand, the function $E|_{]0, \mu_0[}$ is strictly decreasing and $E (\nu) \leq \E (u_\even) \leq E (\mu)$, so $\nu \geq \mu$. On the other hand, it follows from Step (3) that $\nu \leq \mu$. As such, we deduce that $\nu = \mu$. As
\[
\mu
=
\norm{u}_{L^2}^2
=
\norm{u_\even}_{L^2}^2 + \norm{u_\odd}_{L^2}^2
=
\nu + \norm{u_\odd}_{L^2}^2
\]
and $\nu = \mu$, we conclude that $u_\odd \equiv 0$.
\qed

\section{Proof of Theorem \ref{thm:notions}}
\subsection{Proof of \ref{thm:notions:i}.}

~\paragraph{Step 1: preliminaries.}
Let $\mu_0 = \mu_0 (N, \alpha, p) > 0$ be furnished by Lemma \ref{lem:E_is_negative_and_SD}. In view of Theorem \ref{thm:main} \ref{thm:main:i}, we can associate each $\mu \in ]0, \mu_0[$ with the following nonempty set of Lagrange multipliers associated with ground states with mass $\mu$:
\begin{multline*}
\Omega_\mu
=
\left\{
	\omega \in \real:
	\text{there exists} ~
	u_\mu \in \Sp (\mu)
\right.
\\
\left.
	\text{such that} ~
	\E (u_\mu) = E (\mu)
	~ \text{and} ~
	\Action' (u_\mu) = 0
\right\}.
\end{multline*}

\paragraph{Step 2: $\omega < \aev$ for every $\omega \in \Omega_\mu$ and $\mu \in ]0, \mu_0[$.}
We want to show that
\begin{equation}
\label{thm:notions:i:step2:1}
\Omega_\mu \subset \ooi*{- \infty, \aev}
~\text{for every}~
\mu \in ]0, \mu_0[.
\end{equation}
By contradiction, suppose that $0 < \mu < \mu_0$,
$\omega \in \Omega_\mu$ and $\omega \geq \aev$. In particular, we can fix $u \in \Sp (\mu)$ such that
$\E (u) = E (\mu)$ and $\Action'(u) = 0$. On one hand,
\begin{equation}
\label{thm:notions:i:step2:2}
\angles*{\Action' (u), u}_{\Sobolevstar, \Sobolev}
=
h_\alpha (u, u)
+
\omega \mu
+
\norm{u}_{L^p}^p
=
0.
\end{equation}
On the other hand,
\[
h_\alpha (u, u)
+
\omega \mu
\geq
h_\alpha (u, u)
+
\aev \mu
\geq
0
\]
because $\ev = \inf \sigma (\op)$. As such, it follows from \eqref{thm:notions:i:step2:2} that $u \equiv 0$. We obtained a contradiction with the fact that
$\norm{u}_{L^2}^2 = \mu > 0$, so we deduce that \eqref{thm:notions:i:step2:1} holds.

\paragraph{Step 3: $\inf \Omega_\mu \to \aev$ as $\mu \to 0$.}
Let us prove that the following limit holds:
\begin{equation}
\label{thm:notions:i:step3:1}
\inf \Omega_\mu \xrightarrow[\mu \to 0]{} \aev.
\end{equation}
Suppose that $\parens{\mu_n}_{n \in \nat}$
is a sequence such that $\mu_n \to 0$ as
$n \to \infty$ and $\parens{\omega_n}_{n \in \nat}$ is a sequence of Lagrange multipliers such that $\omega_n \in \Omega_{\mu_n}$ for every $n \in \nat$. Let us prove that $\omega_n \to \aev$ as
$n \to \infty$. Given $n \in \nat$, fix $u_n \in \Sp (\mu_n)$ such that $\E (u_n) = E (\mu_n)$ and
$\mathcal{S}_{\omega_n}' (u_n) = 0$.
In particular,
\begin{equation}
\label{thm:notions:i:step3:2}
\angles*{
	\mathcal{S}_{\omega_n}' (u_n), u_n
}_{\Sobolevstar, \Sobolev}
=
h_\alpha (u_n, u_n)
+
\omega_n \mu_n
+
\norm{u_n}_{L^p}^p
=
0.
\end{equation}
As $\ev = \inf \sigma (\op)$, we deduce that
\[h_\alpha (u_n, u_n) + \aev \mu_n \geq 0.\]
Therefore,
\[
\frac{1}{p}
\norm{u_n}_{L^p}^p
\leq
\frac{1}{2} h_\alpha (u_n, u_n)
+
\frac{\aev \mu_n}{2}
+
\frac{1}{p}
\norm{u_n}_{L^p}^p
=
\E (u_n)
+
\frac{\aev \mu_n}{2}.
\]
It follows from Lemma \ref{lem:limit} that
\[
\frac{\E (u_n)}{\mu_n}
+
\frac{\aev}{2}
\xrightarrow[n \to \infty]{}
0,
\]
so we deduce that
\begin{equation}
\label{thm:notions:i:step3:3}
\frac{1}{\mu_n} \norm{u_n}_{L^p}^p
\xrightarrow[n \to \infty]{}
0.
\end{equation}
In view of Lemma \ref{lem:limit} and \eqref{thm:notions:i:step3:3}, we obtain the following limit:
\begin{equation}
\label{thm:notions:i:step3:4}
\frac{h_\alpha (u_n, u_n)}{\aev \mu_n}
\xrightarrow[n \to \infty]{}
- 1.
\end{equation}
Finally, the convergence $\omega_n \to \aev$ as $n \to \infty$ follows from \eqref{thm:notions:i:step3:2}--\eqref{thm:notions:i:step3:4} and we deduce that \eqref{thm:notions:i:step3:1} holds.

\paragraph{Step 4: conclusion.}
It follows from \eqref{thm:notions:i:step2:1} and \eqref{thm:notions:i:step3:1} that we can fix
$\mu_1 = \mu_1 (N, \alpha, p) > 0$
such that
\begin{equation}
\label{thm:notions:i:step4:1}
\Omega_\mu \subset \ooi*{0, \aev}
\quad \text{for every} \quad
\mu \in ]0, \mu_1[.
\end{equation}
Finally, the result follows from an application of Theorem \ref{thm:FOR}.
\qed

\subsection{Proof of \ref{thm:notions:ii}.}
By contradiction, suppose that there exists
$v \in \Sp (\mu)$ such that
\begin{equation}
\label{thm:notions:ii:1}
\E (v) < \E (u).
\end{equation}

Let us prove that
\begin{equation}
\label{thm:notions:ii:step2:1}
h_\alpha (v, v) < 0.
\end{equation}
In view of the equalities
\begin{equation}
\label{thm:notions:ii:step2:1.5}
\norm{u}_{L^2}^2 = \norm{v}_{L^2}^2 = \mu,
\end{equation}
it follows from \eqref{thm:notions:ii:1} that
\begin{equation}
\label{thm:notions:ii:step2:2}
\Action (v) < \Action (u).
\end{equation}
As $u \in \ACP$, we deduce that
\[
0
=
\angles*{\Action' (u), u}_{\Sobolevstar, \Sobolev}
=
h_\alpha (u, u) + \omega \mu + \norm{u}_{L^p}^p.
\]
Therefore,
\begin{equation}
\label{thm:notions:ii:step2:3}
\Action (u)
=
- \frac{p - 2}{2 p}
\norm{u}_{L^p}^p
<
0.
\end{equation}
It follows from \eqref{thm:notions:ii:step2:2} and \eqref{thm:notions:ii:step2:3} that
$\E (v) \leq \Action (v) < 0$. As
$\norm{v}_{L^p} > 0$, we deduce that \eqref{thm:notions:ii:step2:1} holds.

Consider the following function:
\begin{equation}
\label{thm:notions:ii:step3:1}
]0, \infty[ \ni t
\mapsto
\Action (t v) \in \real.
\end{equation}
Clearly, \eqref{thm:notions:ii:step2:1} implies the existence of a unique $t_v > 0$ such that
\[\Action (t_v v) = \min_{s > 0} \Action (s v).\]
It also holds that $t_v v \in \Ne$, where
\[
\Ne
:=
\set*{
	w \in \Sobolev \setminus \set{0}:
	\angles*{\Action' (w), w}_{\Sobolevstar, \Sobolev}
	=
	0
}
\]
denotes the \emph{Nehari manifold} associated with
$\Action$. In view of \eqref{thm:notions:ii:step2:2}, we deduce that
$\Action (t_v v) \leq \Action (v) < \Action (u)$.
This contradicts the fact that
\[\Action (u) = \min_{w \in \Ne} \Action (w).\]
As such, we deduce that \eqref{thm:notions:ii:1} cannot hold.
\qed

\addcontentsline{toc}{section}{References}
\sloppy
\printbibliography
\end{document}